\documentclass[a4paper,11pt]{amsart}
\usepackage[utf8]{inputenc}
\usepackage{lmodern}
\usepackage[T1]{fontenc}
\usepackage[margin=3.3cm]{geometry}
\usepackage{scalerel, stackengine}
\usepackage{graphicx}
\usepackage{soul} 
\usepackage{comment}
\usepackage{array}
\usepackage{enumitem}
\usepackage{amsmath,amssymb,amsfonts,amsthm}
\usepackage{thmtools}
\usepackage{xcolor}
\usepackage{hyperref}
\usepackage{cleveref}
\hypersetup{colorlinks=true,linkcolor=gray,citecolor=gray}
\usepackage[hyperpageref]{backref} 

\theoremstyle{plain}
\newtheorem{lemma}{Lemma}[section]
    \newtheorem{theorem}[lemma]{Theorem}
    \newtheorem{proposition}[lemma]{Proposition}
    \newtheorem{corollary}[lemma]{Corollary}
    
    \newtheorem*{conjecture*}{Conjecture}

\theoremstyle{definition}

    \newtheorem*{problem*}{Problem}
    
\theoremstyle{remark}
    \newtheorem{remark}[lemma]{Remark}

    \newtheorem*{remark*}{Remark}
    \newtheorem*{metaconjecture*}{Metaconjecture}

\numberwithin{equation}{section}

\makeatletter
\newcommand{\customitem}[3][condition]{%
    \item[#3] 
    \begingroup
    \cref@constructprefix{#1}{\cref@result}%
    \protected@edef\@currentlabel{%
    #3}%
    \protected@edef\@currentlabelname{#3}%
    \protected@edef\cref@currentlabel{%
    [#1][][\cref@result]%
    #3%
  }
  \label[#1]{#2}
  \endgroup  
}
\makeatother

\crefname{condition}{condition}{conditions} \crefformat{condition}{Condition~#2#1#3}

\newcommand{\GL}{\operatorname{GL}}
\newcommand{\SL}{\operatorname{SL}}

\newcommand{\Ma}{\operatorname{M}}

\newcommand\Char{\operatorname{char}}

\newcommand{\ZZ}{\mathcal{Z}}
\renewcommand{\O}{\mathcal{O}}
\newcommand{\U}{\mathcal{U}}

\newcommand{\qa}[3]{\left(\frac{#1, #2}{#3}\right)}
\newcommand\PCI{\operatorname{PCI}}

\newcommand\Cay{\operatorname{Cay}}

\newcommand\ot{\otimes}
\newcommand\op{\oplus}

\newcommand\mc{\mathcal}

\stackMath
\newcommand\reallywidehat[1]{%
\savestack{\tmpbox}{\stretchto{%
  \scaleto{%
    \scalerel*[\widthof{\ensuremath{#1}}]{\kern-.6pt\bigwedge\kern-.6pt}%
    {\rule[-\textheight/2]{1ex}{\textheight}}
  }{\textheight}%
}{0.5ex}}%
\stackon[1pt]{#1}{\vstretch{1.5}{\tmpbox}}%
}

\newcommand{\N}{{\mathbb N}}
\newcommand{\Q}{{\mathbb Q}}

\newcommand{\Z}{{\mathbb Z}}

\newcommand\restr[2]{{
    \left.\kern-\nulldelimiterspace 
    {#1} 
    \right|_{#2} 
    }}
\renewcommand{\restr}[2]{
    \ensuremath{
    {#1}_{\mkern2mu {\vrule height 2ex} \mkern2mu {#2}}
    }}
\renewcommand{\restr}[2]{
    \sbox0{$#1$}
    \ensuremath{
    {#1}\raisebox{-1.2ex}{%
        \(\mkern2mu {\vrule height 2.4ex} \mkern2mu {\scriptstyle{#2}}\)}
    }}

  {\left[\begin{smallmatrix}}
  {\end{smallmatrix}\right]}

\newcommand{\Fi}{F} 
\newcommand{\Pnai}[1]{({P}\textsubscript{na\"i})} 

\makeatletter
\newif\ifnobrackets
\renewcommand\@cite[2]{\ifnobrackets\else[\fi{#1\if@tempswa , #2\fi}\ifnobrackets\else]\fi\nobracketsfalse}

\makeatother

\begin{document}

\title[Unit theorem for several ends and applications]{A unit theorem for products of groups with several ends and Applications}
\author{Geoffrey Janssens}
\address{(Geoffrey Janssens) \newline Institut de Recherche en Math\'ematiques et Physique, UCLouvain, 1348 Louvain-la-Neuve, Belgium and \newline Department of Mathematics and Data Science, Vrije Universiteit Brussel,
    Pleinlaan $2$, 1050 Elsene
    \newline E-mail address: \texttt{geoffrey.janssens@uclouvain.be}}

\begin{abstract}
In his $1994$ survey \cite{KleinertSurvey}, Kleinert defined formally and formulated the problem to obtain unit theorems for unit groups of orders in a semisimple algebra $A$. If $A$ is a group algebra $FG$, it boils down to classifying all finite groups $G$ such that the unit groups of most orders in $FG$ belong to a prescribed class $\mc{G}$ of infinite groups. We solve this problem for $\mc{G}$ consisting of the groups which are virtually a direct product of groups whose Cayley graph has more than one end. Subsequently, we obtain two types of applications. A first type being about the existence of torsion-free normal complements and a second about obtaining short and uniform proofs of some of the main results in \cite{JesLea2, JL3, JesPar, Jes}.
\end{abstract}



\thanks{The author is grateful to the Fonds Wetenschappelijk Onderzoek vlaanderen -- FWO (grant 2V0722N) and to the Fonds de la Recherche Scientifique -- FNRS (grant 1.B.239.22) for their financial support.}

\maketitle

\newcommand\blfootnote[1]{%
  \begingroup
  \renewcommand\thefootnote{}\footnote{#1}%
  \addtocounter{footnote}{-1}%
  \endgroup
}
\vspace{-0,3cm}

\section{Unit theorems and the virtual structure problem}\label{subsec:introductionVSP}
\subsection{Historical background of the problem}
The interest in structure theorems for unit groups of orders goes back to the dawn of ring theory via for example Dirichlet's Unit Theorem for ring of integers of number fields and to the pioneering work of Higman \cite{Higman} in case of group algebras. Such results are called `unit theorems' and a concrete meaning was formulated and asked for in Kleinert's $1994$ survey \cite{KleinertSurvey}:
\begin{quote}
``A unit theorem for a finite-dimensional semisimple rational algebra $\Lambda$ consists of the definition, 
in purely group-theoretical terms, of a class of groups $C(\Lambda)$ such that almost all generic unit groups of $\Lambda$ are members of $C(\Lambda)$.''
\end{quote} 
Recall that a generic unit group of $\Lambda$ is a subgroup of finite index in the group of elements of reduced norm $1$ of an order\footnote{In this article an order refers to a $\Z$-order, but the question also makes sense for more general orders.} in $\Lambda$. 

An important feature of orders in semisimple algebras is that their unit groups are commensurable \cite[Lemma 4.6.9]{EricAngel1}. Therefore, in the case that the class $C(\Lambda)$ is closed under commensurability, a generic unit group will be in $C(\Lambda)$ if and only if the unit group of one specific order is in $C(\Lambda)$. For instance, in the case that $\Lambda= \Fi G$ is a group algebra with $F$ a number field, one could choose the order $RG$ with $R$ the ring of integers of $F$. In this setting, Kleinert's question boils down to the `Virtual Structure Problem' which was formulated explicitly for the first time around $2000$ in \cite{JesRioReine}:

\begin{problem*}[Virtual Structure Problem] \label{que:VSP}
Let $\mc{G}$ be a class of groups. 
Classify the finite groups $G$ and the number fields $\Fi$ such that $\U(RG)$ has a subgroup of finite index which belongs to the class $\mathcal{G}$.
\end{problem*}

Up to our knowledge, until quite recently the only classes of groups $\mathcal{G}$ for which the virtual structure problem was known were the following:

\begin{itemize}
    \item $\mc{G}_{fin} = \{ \text{ finite groups } \},$ \cite[Corollary 5.5.8]{EricAngel1}
    \item $\mc{G}_{ab} = \{ \text{ abelian groups } \},$ \cite{Higman}
    \item $\mc{G}_{solv}  = \{ \text{ solvable groups } \},$ \cite[Theorem 2]{Kleinert}
    \item $\mc{G}_{fbf}  = \{ \text{ direct product  of free-by-free groups } \}$ and certain subclasses (\cite{JesLealRio, JesRioReine, PitaRioRuiz, FreebyFree}), 
\end{itemize}

The class $\mc{G}_{fbf}$ was the culmination of various works superseding each other. For instance, prior to $\mc{G}_{fbf}$ and building on \cite{Jes, JesLealRio, LealRio, JesLeal}, Jespers and Del R\`io answered in \cite{JesRioReine} the virtual structure problem for 
\[
\mathcal{G}_{pab} = \Big\{ \prod_i A_{i,1} \ast \cdots \ast A_{i,t_i} \mid A_{i,j} \text{ are finitely generated abelian } \Big\},
\]
where $t_i =1$ is allowed (i.e.\ an abelian factor). 
It turns out that the classification coincides with that of the class of products of free groups (throughout, $\Z$ is considered to be free). 

The answer for $\mc{G}_{fbf}$ already dates back to $2007$ and skepticism grew whether there were other unit theorems. Recently, new classes were achieved through the point of view of geometric group theory. Concretely, the virtual structure problem was answered for classes of groups satisfying some fixed point property such as property (T) or HFA \cite{BJJKT, BJJKT2}.

Remarkably, in all these cases there is also a characterization in terms of the  group algebra. For example, for $F = \Q$ and $\mathcal{G}_{fbf}$, one has that all generic unit groups of $F G$ are virtually in $\mc{G}_{fbf}$ if and only if every simple factor of $F G$ is either a field, a totally definite quaternion algebra or $\Ma_2(K)$, where $K$ is either $\Q$, $\Q(i)$, $\Q(\sqrt{-2})$ or $\Q(\sqrt{-3})$.

\subsection{Contributions to the Virtual Structure Problem}

In this article, we will obtain new unit theorems. Foremost, we answer the virtual structure problem for the class
\[
\mathcal{G}_{\infty} := \Big\{ \prod_i \Gamma_i \mid \Gamma_i \text{ has infinitely many ends } \Big\}.
\]

Recall that given a finitely generated group $\Gamma$, the \emph{number of ends} $e(\Gamma)$ is the infimum of the number of infinite connected components of $\Cay(\Gamma, S) \setminus F$, for $F$ going through finite subsets of $\Cay(\Gamma, S)$.  This number $e(\Gamma)$ does not depend on the choice of a generating set $S$ for $\Gamma$. 

It is well known that for a finitely generated group $\Gamma$, the  number of ends $e(\Gamma)$ is either $0$, $1$, $2$, or $\infty$. By Stallings' theorem \cite{StaPaper,StaBook} a group has infinitely many ends if and only if it can be decomposed as an amalgamated product or an HNN extension, over a finite group.  In fact, we will mainly work with this characterization. It is known that commensurable groups have equal number of ends, see \cite[p.\ 38]{StaBook}. Hence resolving the virtual structure problem for $\mc{G}_{\infty}$ also yields a unit theorem in the sense of Kleinert.

Despite that the class $\mc{G}_{\infty}$ is much larger than the class 
$$\Big\{ \prod_i \Gamma_i \mid \Gamma_i \text{ is a non-cyclic free group} \Big\},$$ we will prove that the virtual structure problem for them has the same answer. 
We will say that a group \emph{virtually belongs to $\mc{G}$} or \emph{is virtually-$\mc{G}$}, if it has a subgroup of finite index belonging to $\mc{G}$. Also, recall that a group $G$ is said to be a \emph{cut group} if $\ZZ (\U ( \Z G))$ is finite.

\begin{theorem}\label{VSP for infinite ends}
Let $G$ be a finite group and $\Fi$ be a number field with ring of integers $R$. 
The following are equivalent. 
\begin{enumerate}[itemsep=1ex,topsep=1ex,label=\textup{(\roman*)}]
\item \label{item:vinf for all}  $\Gamma$ is virtually-$\mc{G}_{\infty}$ for all orders $\O$ in $\Fi G$ and all finite-index subgroups $\Gamma \leq \U(\O)$. 
\item \label{item:vinf for grp ring} $\Fi = \Q$ and $\U(R G)$ is virtually-$\mathcal{G}_{\infty}$. 
\item \label{item:vinf in comp}  All simple components of $\Fi G$ are of the form $\Q(\sqrt{-d})$ with $d \in \mathbb{N}$, $\qa{-a}{-b}{\Q}$ with $a,b \in \N \setminus \{0\}$, or $\Ma_2(\Q)$, and the latter occurs at least once. 
\end{enumerate}
If so, $G$ is a cut group and only $(-1,-1)$ and $(-1,-3)$ can occur for $(-a,-b)$. 
\end{theorem}

Using assertion \ref{item:vinf in comp} of \Cref{VSP for infinite ends} in terms of simple components of $FG$ and \cite[Proposition 1.1]{JesRioReine}, we see that the answer for the class $\mc{G}_{\infty}$ coincide with the answer for the class obtained in \cite[Theorem 1]{LealRio}.

\begin{corollary}\label{coro: classif coincides}
Let $G$ be a finite group and $R$ the ring of integers of a number field. The following are equivalent. 
\begin{enumerate}[itemsep=1ex,topsep=1ex,label=\textup{(\roman*)}]
\item $\U( R G)$ is virtually in $\mc{G}_{\infty}$.
\item $\U (R G)$ is virtually in $\Big\{ \prod_i \Gamma_i \mid \Gamma_i \text{ is a non-cyclic free group} \Big\}$.
\end{enumerate}
\end{corollary}

In \cite{JesLealRio,LealRio} the finite groups satisfying assertion (iii) in \Cref{VSP for infinite ends} have been classified. Therefore \Cref{VSP for infinite ends} indeed answers the Virtual Structure problem for the class $\mathcal{G}_{\infty}$. 

\begin{remark*}
That the groups $G$ such that $\U (R G)$ is virtually a direct product of non-abelian free groups have to be cut follows from the classification obtained in \cite[Theorem 1]{LealRio}, as all components mentioned in \Cref{VSP for infinite ends} have a maximal order with finite center. However the characterization in \Cref{coro: classif coincides} via $\mathcal{G}_{\infty}$ gives a classification-free proof of that fact, since it is a general property of virtually-$\mathcal{G}_{\infty}$ groups to have finite center (see \Cref{facts on number of ends}).
\end{remark*}
\medskip

As mentioned earlier, the  number of ends $e(\Gamma)$ is either $0$, $1$, $2$, or $\infty$. In the remainder of the paper we describe when $e(\U(RG))= 0,2$ or $\infty$.  

First, by definition, $e(\Gamma)= 0$ if and only if $\Gamma$ is finite. Therefore, Higman's theorem \cite[Theorem 1.5.6.]{EricAngel1} translates to:
\begin{equation}\label{zero ends}
    e(\U(\Z G))= 0 \text{ if and only if } \left\lbrace \begin{array}{c}
        G \text{ is abelian with } \exp(G) \mid 4,6\\
        \text{or }\\
          G \cong Q_8 \times C_2^n \text{ for some } n
    \end{array}\right.
\end{equation}
Using \cite[Theorem 4.1]{JJS} one can formulate the variant of \eqref{zero ends} for twisted group rings $\U(R^{\alpha}G)$ over an arbitrary ring of integers $R$.
\medskip

Next we describe the case that $e(\U(RG))= \infty$. In the upcoming result we use the notation $D_{2n}= \langle a,b \mid a^n = 1 = b^2, a^b = a^{-1}\rangle, $ $ Dic_3 = \langle a,b \mid a^6, a^3= b^2, a^b = a^{-1} \rangle $ and $C_4 \rtimes C_4 = \langle a,b \mid a^4 = b^4 = 1, a^b=a^{-1}\rangle$. 

\begin{theorem}\label{th: infinite ends}
Let $G$ be a finite group and $\Fi$ be a number field with ring of integers $R$. Then for a finite index subgroup $\Gamma$ of $\U (R G)$ the following are equivalent:
\begin{enumerate}[itemsep=1ex,topsep=1ex,label=\textup{(\roman*)}]
\item \label{item: infty gamma} $e(\Gamma)=\infty$.
\item \label{item: infty grp ring} $F= \Q$ and $\U (R G)$ is virtually free.
\item \label{item: the grps} $F= \Q$ and $G$ is isomorphic to $D_6$, $D_8$, $Dic_3$, or $C_4 \rtimes C_4$.
\end{enumerate}
\end{theorem}

That $e(\Gamma)=\infty$ is equivalent with the unit group of the group ring being virtually free is readily proven using the methods used in the proof of \Cref{VSP for infinite ends}. The most innovative part of the proof of \Cref{th: infinite ends} concerns the fact that $\U (\Z G)$ is virtually free for exactly the four groups cited. The latter equivalence is already known since \cite{Jes, JesRioReine}, however we will give a new and short proof of this fact using amalgamated product machinery. In particular, \Cref{th: infinite ends} should rather be viewed as an application.

\medskip
Finally, recall that $e(\Gamma) = 2$ if and only if $\Gamma$ has a subgroup of finite index isomorphic to $\Z$, see \cite{Hopf, Freud} or \cite[p.\ 38]{StaBook}. For $\Gamma = \U (RG)$ this happens for only few cases.

\begin{proposition}\label{prop:RGtwoends} 
Let $G$ be a finite group, $\Fi$ a number field and $R$ its ring of integers.
The following are equivalent:
\begin{enumerate}[itemsep=1ex,topsep=1ex,label=\textup{(\roman*)}]
    \item \label{item: gamma 2} $e(\Gamma) = 2$ for some finite-index subgroup $\Gamma$ in the unit group of an order $\mc{O}$ in $\Fi G$. 
    \item \label{item: grpring 2}$\U (RG)$ is $\Z$-by-finite and $F = \Q (\sqrt{-d})$ for $d \in \N$. 
    \item \label{item: groups 2} $G$ is isomorphic to $C_n$ with $n = 5,8$ or $12$ and $F= \Q(\zeta_d)$ with $d \nmid n$. 
\end{enumerate}
\end{proposition}

\begin{remark*}
In light of \Cref{VSP for infinite ends} and \Cref{prop:RGtwoends}, it would be natural to consider the class
\[
\mc{G}_{\neq 1} := \Big\{\prod_i \Gamma_i \mid e(\Gamma_i) \neq 1 \Big\}.
\]
In fact, with similar methods one can prove for a finite group $G$ that
\begin{equation}
\text{$\U(\Z G)$ is virtually-$\mathcal{G}_{\neq 1}$} \iff \text{$\U(\Z G)$  is virtually-$\mathcal{G}_{pab}$}.
\end{equation}
\end{remark*}

\subsection{Further applications to normal complements}

As mentioned earlier, \Cref{th: infinite ends} can be considered as a first application of our methods and the infinite number of ends property. Our next application concerns normal complements. 

After a use of the equivalence \ref{item: infty gamma} $\Leftrightarrow$ \ref{item: infty grp ring} in \Cref{th: infinite ends}, part \ref{item: normal complement} of \Cref{th: main application} below is in fact the main theorem of \cite{Jes}. More precisely, the proof in \cite{Jes} first classifies the finite groups $G$ such that $V(\Z G)$ is virtually free (see \ref{item: the grps} in \Cref{th: infinite ends}). Additionally, the papers \cite{JesLea2, JL3, JesPar} construct case-by-case a normal complement for the trivial units. We provide a new proof which is uniform, in the sense that it does not require to classify all possible $G$ and in particular does not involve a case-by-case study. This is possible thanks to our characterization via $\mc{G}_{\infty}.$
\begin{theorem}\label{th: main application}
Let $G$ be a finite group, $\Fi$ a number field and $R$ its ring of integers. Then the following hold:
\begin{enumerate}[itemsep=1ex,topsep=1ex,label=\textup{(\roman*)}]
\item \label{item: block normal complements} If $\U(RG)$ is virtually-$\mc{G}_{\infty}$, then for each $e \in \PCI(\Fi G)$ such that $\Fi Ge$ is not a division algebra,  there is a free normal subgroup $N_e$ of $\U(RG)e$ such that 
$$V(RG)e \cong N_e \rtimes Ge.$$ 
\item \label{item: normal complement}
If $e(\U(RG)) = \infty$, then there exists a normal free subgroup $F_m$ in $\U(RG)$ such that
$$V(RG) \cong F_m \rtimes G.$$

\item \label{item: block ZC} If $\U(RG)$ is virtually-$\mc{G}_{\infty}$ and $H$ a finite subgroup of $V(RG)$, then $(He)^{\alpha_e} \leq Ge$ for some $\alpha_e \in \U(RGe)$.
\end{enumerate}
\end{theorem}

Part \ref{item: block ZC} of \Cref{th: main application} is a third Zassenhaus conjecture type of statement. In the upcoming work \cite{CJTvsp}, more applications of \Cref{VSP for infinite ends} to the ``blockwise Zassenhaus conjecture'' will be investigated, that is, to the question whether $He$ is conjugated inside $\U(R Ge)$ to a subgroup of $Ge$, for $H$ a finite subgroup of $V(R G)$ and $e$ a primitive central idempotent of $FG$.

\begin{remark}
    It follows from the proof of \Cref{th: main application} and \cite[Corollary 10.14]{JJS} that the subgroup $N_e$ in \ref{item: block normal complements} is generated by the $H$-units recently introduced in \cite{JJS}.
\end{remark}

\vspace{0,1cm}
\noindent \textbf{Acknowledgment.} We thank Robynn Coverleyn for comments on a preliminary version of this article.

\section{Proofs of the statements}\label{sec:proofs}

\subsection{Proof of VSP for products of groups with infinitely many ends}
The remainder of the paper is dedicated to the proof of \Cref{VSP for infinite ends}
We first need to record some general facts on the number of ends and to be virtually-$\mc{G}_{\infty}$.

\begin{lemma}\label{facts on number of ends}
Let $\Gamma_1,\Gamma_2$ and $\Gamma$ be finitely generated groups. 
\begin{enumerate}
    \item If $\Gamma_1$ and $\Gamma_2$ are commensurable, then $e(\Gamma_1)= e(\Gamma_2)$. Furthermore, $\Gamma_1$ is virtually-$\mc{G}_{\infty}$ if and only if $\Gamma_2$ is virtually-$\mc{G}_{\infty}$.
    \item If $e(\Gamma)= \infty$ and $\Gamma$ is commensurable with $\Gamma_1\times \Gamma_2$, then $\Gamma_1$ or $\Gamma_2$ is finite.
    \item If $\Gamma$ is virtually-$\mc{G}_{\infty}$, then $\ZZ(\Gamma)$ is finite. 
\end{enumerate}
The first part in particular applies to $\Gamma_i= \U(\O_i)$ for two orders $\O_i$ in a finite dimensional semisimple algebra $A$. 
\end{lemma}
With `` $\Gamma$ is commensurable with $\Gamma_1\times \Gamma_2$'' we mean the following: $\Gamma_1$ and $\Gamma_2$ are normal in $\Gamma$, the intersection $\Gamma_1 \cap \Gamma_2$ is finite and $\Gamma_1\Gamma_2$ is of finite index in $\Gamma$.

\begin{proof}
That the number of ends is constant on commensurability classes is classical and can be found in \cite[p.\ 38]{StaBook}. Now we prove that being virtually-$\mc{G}_{\infty}$ is also a property shared by commensurable groups. By definition, $\Gamma_1 \cap \Gamma_2$ is a common subgroup of finite index and hence it is enough to obtain the desired statement for the case that $\Gamma_2$ is a subgroup of finite index in $\Gamma_1.$ In that case the $\mc{G}_{\infty}$ subgroup of finite index in $\Gamma_2$ is also of finite index in $\Gamma_1$. Thus assume that $\Gamma_1$ is virtually-$\mc{G}_{\infty}$. Denote by $H = \prod_{j=1}^m H_j$ a subgroup of finite index in $\Gamma_1$ with $e(H_j)=\infty$ for all $1 \leq j \leq m.$ Then $H \cap \Gamma_2= \prod_{j=1}^m (H_j\cap \Gamma_2)$ is of finite index in $\Gamma_2$. Moreover, $H_j\cap \Gamma_2$ is of finite index in $H_j$ and hence by the first part $e(H_j) = e(H_j\cap \Gamma_2)$, which shows that $\Gamma_2$ is virtually-$\mc{G}_{\infty}$ via $H \cap \Gamma_2.$\smallskip

For the second statement note that the combination of the first part and the assumption yields that $e(\Gamma_1\times\Gamma_2)=\infty.$ However, the Cayley graph of a direct product is the cartesian product of the Cayley graphs. Using this one can see that the number of ends of a direct product of infinite finitely generated groups is always one, a contradiction. \smallskip

For the third part, consider a subgroup $H = \prod_{j=1}^m H_j$ of finite index in $\Gamma$ with $e(H_j) = \infty$ for $1 \leq j \leq m.$ First remark that $\ZZ(H_j)$ is finite for each $j$. Indeed by Stallings' theorem \cite{StaPaper,StaBook} $H_j$
can be decomposed as a non-trivial amalgamated product or HNN extension over a finite group $C$. In such decomposition one necessarily has that $\ZZ(H_j) \leq C$. Therefore also $\ZZ (H_j)$ is finite, as claimed. Consequently, $\ZZ(H)$ is finite. Now, since $H \cap \ZZ (\Gamma) \leq \ZZ(H)$ is of finite index in $\ZZ (\Gamma)$ we obtain that $\ZZ (\Gamma$, as desired.\smallskip

Finally, unit groups of orders are commensurable by \cite[Lemma 4.6.9]{EricAngel1} and hence indeed the first statement applies to them. 
\end{proof}

Next we need the following lemma, which is a generalization of \cite[Proposition 4.5]{JesRioReine}. 

\begin{lemma}\label{when simple has infinitly ends}
Let $G$ be a finite group, $D$ be a finite dimensional division algebra over $\Fi$ with $\Char (F)=0$, different\footnote{This condition is not necessary, i.e\ the number of ends of $GL_2(\O)$ for $\O$ an order in $\qa{-2}{-5}{\Q}$ is not infinite. However including this case would make the proof unnecessarily lengthy.} of $\qa{-2}{-5}{\Q}$, and suppose $\Ma_n(D)$ with $n\geq 2$ is a simple component of $F G$. If $\mathcal{O}$ is an order in $\Ma_n(D)$, then $e(\U( \O)) = \infty$ if and only if $n=2$ and $D=F=\Q$.
\end{lemma}
\begin{proof} 
Suppose  $e(\U( \O)) = \infty$. \Cref{facts on number of ends} allows to assume without lose of generality that $\O$ is a maximal order in $\Ma_n(D)$. It is well known that in that case $\O \cong \Ma_n(\mathcal{O}_{max})$ with $\mathcal{O}_{max}$ a maximal order in $D$. 

Next recall that any group with infinitely many ends has finite center (as central elements need to be in the subgroup over which the amalgam and HNN are constructed, which is finite in this case). Therefore,  $\SL_n(\O_{max})$ has finite index in $\GL_n(\O_{max})$ and hence $\SL_n(\O_{max})$ also has infinitely many ends. This implies that $\SL_n(\O_{max})$ has $S$-rank $1$, with $S$ the set of infinite places, as otherwise it has hereditarily Serre's property FA (even property $(T)$ \cite{MargulisBook,ErsJai}) and hence can not have a non-trivial amalgam or HNN splitting.

The S-rank being one means that $n= 2$ and $D$ is either $\Q(\sqrt{-d})$, with $d \geq 0$ or $\qa{-a}{-b}{\Q}$ with $a,b$ strictly positive integers (see for instance \cite[Theorem 2.10.]{BJJKT}). 
Furthermore it was proven in \cite{EKVG} that the condition that $\Ma_2(D)$ is a component of a group algebra yields that $d \in \{0, -1,-2,-3\}$ and $(a,b) \in \{(1,1),(1,3),(2,5)\}$. All these division algebras are (right norm) Euclidean and due to this have a unique maximal order (see \cite[remark 3.13]{BJJKT}), which we denote $\O_D$. By assumption $(a,b) = (2,5)$ does not occur. Now, following \cite[Theorem 5.1]{BJJKT} $\GL_2(\O_D)$ has property FA except if $D= \Q$ or $\Q(\sqrt{-2})$. In case of $D= \Q(\sqrt{-2})$ one can use the amalgam decomposition of $\SL_2(\Z[\sqrt{-2}])$ given in \cite[Theorem 2.1]{FroFin} to see that the group does not admit a splitting over a finite group. Finally, $\GL_2(\Z) = D_8 \ast_{C_2 \times C_2} D_{12}$ and hence $e(\GL_2(\Z)) = \infty$, finishing the proof.
\end{proof}

We now proceed with the proof of the main result.

\begin{proof}[Proof of \Cref{VSP for infinite ends}]
First observe that \Cref{facts on number of ends} entails that if $\Gamma \leq \U (\O)$ is virtually-$\mathcal{G}_{\infty}$ for some order $\O$ in $FG$, then also all other finite index subgroups of the unit group of any other order in $FG$ is virtually-$\mathcal{G}_{\infty}$. In particular we obtained:\smallskip

\noindent \underline{Claim 0:} Statement $\ref{item:vinf for all}$ holds if and only if $\U(RG)$, with $R$ the ring of integers in $\Fi$, is virtually-$\mathcal{G}_{\infty}$. \smallskip

The above claim yields the implication $\ref{item:vinf for grp ring} \Rightarrow \ref{item:vinf for all} $ and the converse follows once we prove that virtually-$\mathcal{G}_{\infty}$ forces $\Fi$ to be the field of rational numbers. 

Next, for the remaining of the proof fix a Wedderburn-Artin decomposition $F G = \bigoplus_{i=1}^q \Ma_{n_i}(D_i)$. Furthermore, fix a maximal order $\O_i$ in $D_i$ for each $i$ such that $\U (RG)$ is a subgroup of $\prod_{i=1}^q \GL_{n_i}(\O_i).$

By our starting observation one has that $\U( R G)$ is virtually-$\mathcal{G}_{\infty}$ if and only if $\prod_{i=1}^q \GL_{n_i}(\O_i)$ is. Now if $\Ma_{n_i}(D_i)$ is either a field or a totally definite quaternion algebra, then $\SL_1(\O_i)$ is finite by \cite[Proposition 5.5.6 ]{EricAngel1}. Also recall that $\Fi$ is in the center of every simple component of $FG$. Hence, if $\Ma_2(\Q)$ is a simple component, then $F= \Q.$ We now obtain the implication \ref{item:vinf in comp} $\Rightarrow$ \ref{item:vinf for grp ring} by combining all the previous with \Cref{when simple has infinitly ends}.

Now suppose that \ref{item:vinf for all} holds. Equivalently, by claim $0$, suppose that $\U (R G)$ is virtually-$\mathcal{G}_{\infty}$ and therefore also $\prod_{i=1}^q \GL_{n_i}(\O_i)$.\smallskip

\noindent \underline{Claim 1:} $\ZZ (\U ( R G))$ and $\ZZ(\O_i)$ are finite for all $1\leq i\leq q$. \smallskip

 This follows from \Cref{facts on number of ends} and the virtually-$\mathcal{G}_{\infty}$ property.\smallskip

Claim $1$ implies that $\SL_{n_j}(\O_j)$ is of finite index in $\GL_{n_j}(\O_j)$ for all $1 \leq j \leq q$. In particular, $e(\GL_{n_j}(\O_j))= e(\SL_{n_j}(\O_j)).$ \smallskip

\noindent \underline{Claim 2:} $e(\SL_{n_j}(\O_j))= \infty$ for all $j$ such that $M_{n_j}(D_j)$ is different of a field or totally definite quaternion algebra (e.g.\ all $j$ for which $n_j \geq 2$). \smallskip

Let $H = \prod_{j=1}^m H_m$ be a subgroup of finite index in $\U(RG)$ such that $e(H_j) = \infty$ for each $1 \leq j \leq m.$ Denote $S_j := \SL_{n_j}(\O_j) \cap H$ which is of finite index in $\SL_{n_j}(\O_j)$, hence it is enough to proof that $e(S_j)=\infty$. Let $p_k$ be the projection of $H$ on $H_k$. Fix some $j$ as in the statement of claim $2$. The condition is equivalent with saying that $\SL_{n_j}(\O_j)$ is infinite \cite{Kleinert}. In particular there exists some $k$ such that $p_k(S_j)$ is infinite\footnote{Otherwise $S_j$ would be finite and hence also the overgroup of finite index $\SL_{n_j}(\O_j)$.}. For such $k$ we will now prove that $|p_k(\prod_{i \neq j} S_i)| < \infty$. For this consider $S := S_j \times  \prod_{i \neq j} S_i$ which by the first claim is of finite index in $H$. Therefore $p_k(S)$ is of finite index in $H_k$ and hence $e(p_k(S)) = \infty$. Now note that $p_k(S)$ is commensurable with $p_k(S_j) \times p_k(\prod_{i \neq j} S_i)$. Indeed, the two subgroups clearly commute, are normal in $p_k(S)$ and $p_k(S_j) \cap p_k(\prod_{i \neq j} S_i) \subseteq \ZZ(p_k(S))$ which is finite since $p_k(S)$ has infinitely many ends. Thus we may apply\footnote{Instead of that statement one could have used the well known \cite[4.A.6.3.]{StaBook} saying that infinite finitely generated normal subgroups of a group with infinitely many ends need to have finite index.} \Cref{facts on number of ends} to conclude that $|p_k(\prod_{i \neq j} S_i)| < \infty$.

Now consider the set $\mathcal{I}_j := \{ k \mid |p_k(S_j)| < \infty \}.$ From the previous it follows that if $k \in \{1, \ldots, q\} \setminus \mathcal{I}_j$, then $p_k(S_j)$ is of finite index in $H_k$. Hence $S_j/ \big( S_j \cap \prod\limits_{i \in \mathcal{I}_j} H_i \big)$ is a subgroup of finite index in $\prod_{k \notin \mathcal{I}_j} H_k$. As the quotient was with a finite subgroup, we obtain that $S_j$ is virtually-$\mathcal{G}_{\infty}$ and hence $\SL_{n_j}(\O_j)$ also. However, under the conditions stated in claim $2$, $\SL_1$ is virtual indecomposable \cite[Theorem 1]{KleRio}. Therefore $\SL_{n_j}(\O_j)$ in fact is even virtually a group with infinitely many ends and so in fact $e(\SL_{n_j}(\O_j))= \infty$, finishing the proof of claim $2$.\vspace{0,1cm}

 Altogether: Claim $1$ says that $\mc{Z}(\O_i)$ is finite and consequently $\SL_{n_j}(\O_j)$ is of finite index in $\GL_{n_j}(\O_j)$ for all $j$. In particular, $e(\GL_{n_j}(\O_j))= \infty$ if $n_j \geq 2.$ Now \Cref{when simple has infinitly ends} implies that $n_j =2$, i.e.\ no higher matrix algebras appear in the decomposition of $FG.$ In such a case no $\qa{-2}{-5}{\Q}$ component arises. Indeed, following \cite[table appendix]{BJJKT} such a component can only arise if $F = \Q$ and $G$ maps onto one of the groups with SmallGroupID [40,3], [240,89] or [240,90]. But a direct verification, e.g.\ via the Wedderga package in GAP, shows that these groups all have matrix components of reduced degree at least $3$. 

 Consequently, \Cref{when simple has infinitly ends} says that all matrix components of $FG$ must be isomorphic to $\Ma_2(\Q)$ and in particular $F= \Q$ (as $\Fi$ is contained in the center of every simple component). As pointed out earlier, together with claim $0$ this finishes the proof the equivalence \ref{item:vinf for all} $ \Leftrightarrow $ \ref{item:vinf for grp ring}

 Furthermore, by \cite[Theorem 2.10. \& Proposition 6.11.]{BJJKT}, if $\Q G e$ is a division algebra $D$ for some primitive central idempotent $e$ of $\Q G$ then $D$ is $\Q(\sqrt{-d})$ with $d\in\Z_{\geq 0}$ or it is a totally definite quaternion algebra with center $\Q$. In summary, we obtained that all components of $\Q G$ are of the desired form, yielding the remaining implication \ref{item:vinf for all} $\Rightarrow $ \ref{item:vinf in comp}.\smallskip

 Finally, that only the parameters $(-1,-1)$ and $(-1,-3)$ appear is due to \cite[Theorem 11.5.14]{Voight} saying that else $\U( \O)$ is cyclic for any order $\O$ in $\qa{-a}{-b}{\Q}$. In those cases $Ge \leq \U(\Z Ge)$ would have an abelian $\Q$-span and thus $\Q Ge \neq \qa{-a}{-b}{\Q}$, a contradiction.
 \end{proof}

 \subsection{Proofs VSP for a given amount of ends}

Finally, we characterize the case that $\mc{U}(RG)$ is virtually a group with infinitely many ends, and not only a product of such.

\begin{proof}[Proof of \Cref{th: infinite ends}]
Let $FG \cong \prod_{i=1}^q \Ma_{n_i}(D_i)$ be the Wedderburn-Artin decomposition of $FG$. Further consider maximal orders $\mc{O}_i$ in $D_i$. Thanks to the commensurability of unit groups of orders and of the number of ends, one has that $\U(RG)$ is a subgroup of finite index in $\prod_{i=1}^q \GL_{n_i}(\O_i)$ and $e(\Gamma)= e(\U(R G)) = e(\prod_{i=1}^q \GL_{n_i}(\O_i)).$ However, the Cayley graph of a direct product is the cartesian product of the Cayley graphs. Using this we see that $e(Q\times P) = 1$ for any finitely generated infinite groups $P,Q$. Therefore $e(\U(R G)) = \infty$ if and only if $e(\GL_{n_{i_0}}(\O_{i_0})) = \infty$ for exactly one $i_0$ and the other factors are finite. In light of \Cref{when simple has infinitly ends} and \cite[Theorem 2.10.]{BJJKT} this happens exactly when $F G$ has exactly one $\Ma_2(\Q)$ component and all the others are $\Q, \Q(\sqrt{-d})$ or $\qa{-a}{-b}{\Q}$. In particular, $F$ must be $\Q$. Furthermore, since $\GL_2(\Z)$ is virtually free we see that in those cases $\U(R G)$ is indeed virtually free. This concludes the equivalence \ref{item: infty gamma} $\Leftrightarrow$ \ref{item: infty grp ring} and moreover gives a description in terms of the isomorphism type of the simple factors of $FG.$ \medskip

Suppose that $G$ is isomorphic to $D_6, D_8, Dic_3$ or $C_4 \rtimes C_4$. It easily verified that they have the following Wedderburn-Artin decomposition:
\begin{equation*}
\begin{array}{lcl}
    \Q[D_6] &\cong& \Q^{\times 2} \times \Ma_2(\Q) \\
    \Q[D_8]  & \cong &\Q^{\times 4} \times \Ma_2(\Q)\\
    \Q[Dic_3] & \cong & \Q^{\times 2} \times \Q(i) \times \qa{-1}{-3}{\Q} \times \Ma_2(\Q) \\
    \Q[C_4 \rtimes C_5] &\cong & \Q^{\times 4} \times \Q(i)^{\times 2} \times \qa{-1}{-1}{\Q} \times \Ma_2(\Q)
\end{array}
\end{equation*}
By the description obtained earlier via simple component of $FG$, we obtain that for all the groups above $e(\U(\Z G))=\infty,$ yielding \ref{item: the grps} $\Rightarrow$ \ref{item: infty grp ring}.  It remains to prove that these groups are the only one where simple components of $\Q G$ satisfies the necessary restrictions \smallskip

Recall that the unit group of the maximal orders of  $\qa{-1}{-1}{\Q}$ and $\qa{-1}{-3}{\Q}$ are respectively $\SL(2,3) \cong Q_8 \rtimes C_3$ and $Dic_3$. As $\Q[\SL_2(\mathbb{F}_3)]$ has a component $\Ma_3(\Q)$, we have that $G$ can not map onto $\SL_2(\mathbb{F}_3)$. Since $\SL_2(\mathbb{F}_3)$ has a unique isomorphism type of non-abelian subgroup, namely $Q_8$, the group $G$ must map onto $Q_8$ if $\qa{-1}{-1}{\Q}$ is a simple factor of $FG$. Now recall that $\U (\Z G)$ is contained in $\prod_{i=1}^q \GL_{n_i}(\O_i)$ with $\GL_{n_i}(\O_i)$ isomorphic to $\GL_2(\Z) \cong D_{12} \ast_{C_2 \times C_2} D_8$ or a maximal order in a quaternion algebra or field mentioned earlier. Then, altogether we have that 
\begin{equation}\label{eq: amalgam of unit grp with 1 comp}
\U(\Z G) \text{ is a subgroup of finite index in } (D_8 \times U) \ast_{C_2 \times C_2 \times U} (D_{12} \times U)
\end{equation}
where $U = A \times Q_8^{s} \times Dic_3^{t}$ for some $s,t$ and with $A$ abelian with $\exp(A) \mid 4,6$.  Using the description of torsion subgroups in amalgamated products we know that, up to conjugation, $G$ is a subgroup of $C_2 \times C_2 \times U$ or it contains transversal elements in the factor $D_8$ or $D_{12}$. 

First suppose that $G$ is conjugated to a subgroup of the amalgamated part $C_2\times C_2 \times U$. Recall that all subgroups of $Dic_3$ and $Q_8$ are cyclic. From this it is readily concluded that the only way to have exactly one matrix component $\Ma_2(\Q)$ is for $G$ to be $Dic_3$. 

Next,  suppose that $G$ is not conjugated to a subgroup of the amalgamated part. Then $\U(\Z G)$ contains a non-trivial amalgamated product $G \ast_{G \cap (C_2 \times C_2 \times U)} (D_m \times U)$ with $m=8$ or $12$. Now \cite[Proposition 2.7]{JTT} yields an irreducible representation $\rho$ of $G$ of reduced degree at least $2$ such that $\ker(\rho)$ is contained in $G \cap (C_2 \times C_2 \times U)$. By assumption on the components of the group algebra, the representation co-restricts to a morphism $G \rightarrow \GL_2(\Z).$ Note that the image $\rho(G)$ needs to be $D_6$ or $D_8$ as $D_{12}$ has two matrix components. From all this we can infer that in order for $\Q G$ to not have more than one matrix component, the intersection $G \cap (C_2 \times C_2 \times U)$ needs to be a central subgroup of order $2$. Thus $G$ is a central extension of $D_6$ or $D_8$ with a $C_2$. A look at the classification of groups of order $12$ and $16$ shows that $G$ is isomorphic to either $Dic_3$ or $C_4 \times C_4$, finishing the proof.
\end{proof}

Next we prove the characterization of when a unit group has exactly two ends.
\begin{proof}[Proof of \Cref{prop:RGtwoends}]
As recalled in the proof of \Cref{when simple has infinitly ends}, $e(\Gamma_1)= e(\Gamma_2)$ if $\Gamma_1$ and $\Gamma_2$ are commensurable (see \cite[p.\ 38]{StaBook}). Moreover, the unit group of two orders are commensurable \cite[lemma 4.6.9]{EricAngel1}. Therefore, $e(\Gamma) = 2$ if and only if $e(\U (RG)) = 2$ for $R$ the ring of integers of $\Fi$. In particular, (ii) implies (i). Furthermore, since for general finitely generated groups $\Gamma$ one has that $e(\Gamma)=2$ if and only if $\Gamma$ is $\Z$-by-finite, we obtain that statement \ref{item: gamma 2} implies that $\U(RG)$ is $\Z$-by-finite.

We will now prove simultaneously that $\U(RG)$ to be $\Z$-by-finite implies the restrictions on $F$ and $G$ mentionned in \ref{item: grpring 2} and \ref{item: groups 2}, which would finish the proof. To do so, recall a result by Kleinert \cite{Kleinert} (or \cite[Corollary 5.5.7]{EricAngel1}) saying that $\U(R G)$ is abelian-by-finite if and only if all the simple components of $F G$ are either fields or totally definite quaternion algebras. In particular $FG$ has no non-trivial nilpotent elements in which case \cite[Theorem 2.6]{SehNilp} tells that $G$ is either abelian or $G \cong Q_8 \times C_2^m \times A$ with $m \geq 0$ and $A$ an abelian group of odd order. We first consider the case that $G \cong Q_8 \times C_2^m \times A$. Then 
$$F G \cong F[Q_8 \times C_2^m] \ot_F FA \cong (4m \, F \op m \qa{-1}{-1}{F}) \ot_{F} FA.$$

We now see that in order to obtain a {\it single} copy of $\Z$ in $\U (R G)$ that this will have to come from a component of $F A$. However this component will appear at least $4$ times and hence such groups are never $\Z$-by-finite.

Now suppose that $G$ is abelian. By the theorem of Perlis-Walker \cite[Theorem 3.5.4]{PolSeh} 
\begin{equation}\label{Perlis-walker}
F G \cong \bigoplus_{d \mid |G|} a_d \frac{[\Q(\zeta_d): \Q]}{[F(\zeta_d):F]}F(\zeta_d)
\end{equation}
with $a_d$ the number of different cyclic subgroups of order $d$. Now denote by $R_{F,d}$ the ring of integers of $F(\zeta_d)$ and recall that by Dirichlet Unit theorem \cite[Theorem 5.2.4]{EricAngel1} the rank of the finitely generated abelian group $\U(R_{F,d})$ is $n_1 + n_2 - 1$ with $n_1$ the number of real embeddings of $F(\zeta_d)$ and $r_2$ the number of pairs of complex embeddings. Note that the rank of $\U(R_{F,d})$ is at least the one of the unit group of the ring of integers of the cyclotomic field $\Q (\zeta_d).$ The latter rank is well-known to be $\frac{\varphi(d)}{2}-1$. A direct computation yields that $\varphi(d) \leq 4$ if and only if $d \in \{2,3,4,5,6 8, 10, 12 \}$ with equality only for $\{ 5,8,10,12 \}$. This combined with \eqref{Perlis-walker} we see that we have {\it exactly one} copy of $\Z$ if and only if $G$ is isomorphic to $C_n$ with $n=5,8$ or $12$, $\text{rank } \U(R_{F,n}) = \text{rank } \U(R_{\Q,n})=1$ and $\text{rank } \U (R) = 0.$ The latter means that $\Fi$ is either $\Q$ or an imaginary quadratic extension of $\Q$ and hence yields the implication \ref{item: gamma 2} $\Rightarrow$ \ref{item: grpring 2}. For the restriction on $\Fi$ written in \ref{item: groups 2} note that due to the values of $n$ the field $\Q(\zeta_d)$ with $d \nmid n$ is either $\Q$ or an imaginary quadratic extension and hence $\text{rank } \U (R) = 0.$ Furthermore, $F(\zeta_n) = \Q(\zeta_n)$ and hence  $\text{rank } \U(R_{F,n}) = \text{rank } \U(R_{\Q,n})$ as desired.
\end{proof}

\subsection{Proof of \Cref{th: main application}}

For efficiency reasons we will first prove the Zassenhaus conjecture type of property for the projections to the simple components of $FG$.

\begin{proof}[Proof of \ref{item: block ZC} of \Cref{th: main application}]
Supppose that $\U (RG)$ is virtually-$\mc{G}_{\infty}$. Let $H$ be a finite subgroup of $V(R G)$ and $e \in \PCI(FG)$ and $H \leq V(R G)$. Observe that the theorem of Cohn-Livingstone \cite[Corollary 2.7]{AngelSurvey} implies that
\begin{equation}\label{order diving in projection}
|He| \mid |Ge|
\end{equation}

Indeed, denote by $\Phi: \Z [G] \rightarrow \Z [Ge]$ the ring epimorphism associated to $e$ which satisfied $\Phi(V(\Z[G])) \subseteq V(\Z[Ge])$. Thus $|\Phi(H)| \mid |Ge|$. To prove that $|He| \mid |\Phi(H)|$ consider the ring epimorphism $\pi \colon \Z[G] \rightarrow \Z[G]e: x \mapsto xe.$ One has that $\omega(G,N) \subseteq \ker(\pi)$. It follows that $\pi$ factorises through $\Z[Ge]$ i.e. there exists a unique morphism $\sigma \colon \Z[Ge] \rightarrow \Z[G]e$ such that $\pi = \sigma \circ \Phi.$ In particular, the order of $He = \pi(H)=\sigma( \Phi(H))$ divides $|\Phi(H)|$, as desired.\smallskip

Now we consider the possible isomorphism types for $FGe$ separately. First suppose that $FGe$ is a division algebra, i.e. $\Q(\sqrt{-d})$ or $\qa{-a}{-b}{\Q}$ by \Cref{VSP for infinite ends}. Then the unit group of its unique maximal order is finite \cite{Kleinert} and in particular also $\U(RGe)$ is finite. Hence if $FGe$ is a field, then the latter unit group is cyclic and thus $He \leq Ge$ due to \eqref{order diving in projection}. If $F Ge \cong \qa{-a}{-b}{\Q}$, then by \cite[Theorem 11.5.14]{Voight} $\U (R Ge)$ is cyclic except if (1) $(a,b)= (-1,-1)$ and $RGe$ the Lipschitz or Hurwitz order (which is maximal) or (2) $(a,b) = (-1,-3)$ and $R Ge$ is the maximal order\footnote{A short concrete summary of the above orders can also be found before and after Theorem 3.14. of \cite{BJJKT}.} of $\qa{-1}{-3}{\Q}$. In case of the maximal orders, one must have that $Ge$ is their unit group and hence $He \leq Ge$. If $R Ge$ is the Lipschitz order, then $Ge \cong Q_8$ and $He$ is some subgroup of the unit group of the Hurwitz quaternions, i.e. a subgroup of $\SL_2(\mathbb{F}_3)\cong Q_8 \rtimes C_3$. Via \eqref{order diving in projection} we see that in fact $He \leq Ge$, as claimed.\smallskip

It remains to consider the components where $FGe \cong \Ma_2(\Q).$ In that case, 
$\U(RG)e$ is a subgroup of finite index in $\GL_2(\Z) \cong D_8 \ast_{C_2\times C_2}D_{12}.$ Therefore there exists some $\alpha_e \in \U (R Ge)$ such that $\alpha_e^{-1}Ge \alpha_e$ is a subgroup of $D_8$ or $D_{12}$. Since the $\Q$-span of $Ge$  is $\Ma_2(\Q)$ we get that $\alpha_e^{-1}Ge \alpha_e= D_6, D_{12} \text{ or } D_8$. In particular $|Ge|$ determines uniquely its isomorphism type.  The same holds for $He$ if $He$ is non-abelian. Therefore, by \eqref{order diving in projection}, we have the desired statement in that case. Suppose now that $|He| = 4$. If $He \cong C_4$, then $He$ is up to conjugation a subgroup of $D_8$ and due to the dividing of the orders, again $He \leq Ge$ after conjugation. If it is an elementary abelian $2$-group, then it is uniquely defined in both $D_{12}$ and $D_8$. As this subgroup is amalgamated we are done.

\end{proof}

\begin{proof}[Proof of \ref{item: block normal complements} and \ref{item: normal complement} of \Cref{th: main application}]
Suppose that $U(RG)$ is virtually-$\mc{G}_{\infty}$. For each $e\in \PCI(FG)$ consider a maximal order $\mc{O}_e$ of $FGe$ containing $RGe$. One has that $V(RG)$ is of finite index in $\prod_{e\in \PCI(FG)} \U(\mc{O}_e).$ Therefore, $V(RG)e$ is of finite index in $\U(\mc{O}_e)$.

Now suppose that $FGe$ is not a division algebra and hence $FGe \cong \Ma_2(\Q)$ by \Cref{VSP for infinite ends}. Recall that $\Ma_2(\Z)$ is the unique maximal order of $\Ma_2(\Q)$. Therefore $V(RG)e$ is of finite index in $\GL_2(\Z)\cong D_8 \ast_{C_2 \times C_2} D_{12}.$ This entails that $Ge$ is conjugated to $D_{2\ell}$ for $\ell=3,4$ or $6$. Moreover, by part \ref{item: block ZC} all torsion subgroups in $V(RG)e$ are conjugated to a subgroup of $Ge$. Combining these facts with \cite[Proposition 8.2]{JJS} and Karrass-Solitar's description \cite[Theorem 5]{KS} of subgroups of an amalgam one can deduce that $A \ast_C B$, it follows that $V(RG)e \cong \langle F_{m_e}, C_2 \times C_2\rangle \ast_{C_2 \times C_2} Ge$
for some free subgroup $F_{m_e}$ with $m_e$ given in \cite[Corollary pg 247]{KS}. From this follows that $V(RG)e \cong \langle F_{m_e}^G \rangle \rtimes G$, as desired.

Now suppose that $e(\U(RG))= \infty.$ By \Cref{th: infinite ends} this means that there is exactly one non-division component of the form $\Ma_2(\Q)$ and the other components are of the form $\Q(\sqrt{-d})$ with $d \in \N$ or $\qa{-a}{-b}{\Q}$. As proven in the proof of \Cref{th: infinite ends}, see \eqref{eq: amalgam of unit grp with 1 comp}, we have that  $V(R G)$ is a subgroup of finite index in $(D_8 \times U) \ast_{C_2 \times C_2 \times U} (D_{12} \times U)$ with $U = A \times Q_8^{s} \times Dic_3^{t}$ for some $s,t$ and with $A$ abelian with $\exp(A) \mid 4,6$. An analogue reasoning as for \ref{item: block normal complements} now yields the desired conclusion.
\end{proof}

\bibliographystyle{plain}
\bibliography{VSPEnds}

\end{document}